\documentclass[12pt]{article} 
\usepackage{amssymb, amsmath, amsthm}
\usepackage{geometry} 
\usepackage{amsfonts} 
\usepackage{color}
\usepackage{cancel}
\usepackage{mathrsfs}
\usepackage{yfonts}
\usepackage{accents}
\usepackage{commath}
\usepackage{relsize}
\usepackage[pdftex]{graphicx} 
\usepackage[utf8]{inputenc}
\usepackage{cite}
\usepackage{multicol}
\usepackage{tabularx}
\usepackage{multirow}
\usepackage{float}
 \usepackage{relsize}
\usepackage{array}
\usepackage{adjustbox}
\usepackage{tikz}
\usepackage{comment}
\usepackage{color}

\newcommand{\sig}{\sigma_{\e}}
\newcommand{\Gtau}{\sqrt{G_{\tau}} d\xi d\tau}

\newcommand{\G}{\sqrt{G_{0}} d\xi}
\newcommand{\Gt}{\sqrt{G_{0}} d\xi d\tau}
\newcommand{\intg}{\int_{\Gamma}}
\newcommand{\pa}{\partial}

\newcommand{\intp}{\int^{\e}_{0}\int_{\Gamma}}

\providecommand{\norm}[1]{\l#1\|}

\newcommand{\e}{\varepsilon}

\newcommand{\ct}[1]{\langle {#1}\rangle \lower.3ex\hbox{$_{t}$}}
\newcommand{\lt}[1]{[ {#1}] \lower.3ex\hbox{$_{t}$}}

\usepackage{bbm}

\newtheorem{thm}{Theorem}[section]
\newtheorem{lem}[thm]{Lemma}

\title{\LARGE{\bf Asymptotic behavior for the principal eigenvalue of a reinforcement problem}\thanks{This research was partially supported by the Grant-in-Aid for Scientific Research (B) (\#26287020) and Challenging Exploratory Research (\#16K13768) of Japan Society for the Promotion of Science.}}

\author{Toshiaki Yachimura\thanks{
Research Center for Pure and Applied Mathematics, Graduate
School of
Information Sciences, Tohoku University, Sendai 980-8579, Japan.
{\em Electronic mail address:}
yachimura@ims.is.tohoku.ac.jp}}
\date{}

\begin{document}
\maketitle

\begin{abstract}
In this paper, we consider the asymptotic behavior for the principal eigenvalue of an elliptic operator with piecewise constant coefficients. This problem was first studied by Friedman in 1980. 
We show how the geometric shape of the interface affects the asymptotic behavior for the principal eigenvalue. This is a refinement of the result by Friedman.  
\end{abstract}

\bigskip

\noindent{2010 {\it Mathematics Subject classification.} 35J20, 49R05}
\bigskip

\noindent {\it Keywords and phrases: eigenvalue problem, two phase, transmission condition, reinforcement problem, domain perturbation} 

\section{Introduction and main result}
In this paper, we study a two-phase eigenvalue problem and we investigate the asymptotic behavior for the principal eigenvalue. First we introduce some notations. Let $\Omega \subset \mathbb{R}^n$ $(n \geqslant 2)$ be a bounded domain with smooth and connected boundary $\Gamma$. For sufficiently small $\e > 0$, put 
\begin{equation*}
\Sigma_{\e} = \left\{ x \in \mathbb{R}^n \,\, | \,\, x = \xi + \tau \nu_{\Gamma}(\xi) \,\,\, \text{for} \,\,\, \xi \in \Gamma, 0 < \tau < \e \right\}, \quad \Omega_{\e} = \Omega \cup \Sigma_{\e} \cup \Gamma,
\end{equation*}
where $\nu_{\Gamma}$ denotes the outward unit normal vector to the $\Gamma$, see Figure $1$. 
We consider the two-phase eigenvalue problem on $\Omega_{\e}$ as follows:
\begin{equation}\label{P}
\begin{cases}
-\mathrm{div} \left(Q_{\e} \nabla \Phi \right) = \lambda \Phi \hspace{-0.1cm} &\text{in} \,\, \Omega_{\e}, \\
\displaystyle \Phi = 0 \, &\text{on} \, \partial \Omega_{\e},
\end{cases}
\end{equation}
where $Q_{\e} = Q_{\e}(x) \left( x \in \Omega_{\e} \right)$ is a piecewise constant function given by 
\begin{equation}
Q_{\e}(x) = \begin{cases}
1, \quad &x \in \Omega, \\
\sig, \quad &x \in \overline{\Sigma}_{\e},
\end{cases}
\end{equation}
where $\sig = \alpha \e$ and $\alpha$ is a positive parameter. 

We consider the problem \eqref{P} in a weak sense, namely, $\lambda \in \mathbb{C}$ is an eigenvalue of \eqref{P} if there exists $\Phi \in H^{1}_{0}(\Omega_{\e})$ such that $\Phi \not \equiv 0$ and 
for any $\varphi \in H^{1}_{0}(\Omega)$, 
\begin{equation}\label{eigenfunction}
\int_{\Omega} \nabla \Phi \cdot \nabla \varphi \, dx + \sig \int_{\Sigma_{\e}} \nabla \Phi \cdot \nabla \varphi \, dx = \lambda \int_{\Omega_{\e}} \Phi \varphi \, dx. 
\end{equation}
By a standard argument of self-adjoint operators, the eigenvalues of \eqref{P} are non-negative real numbers and the set of all eigenvalues is discrete. 
Let $\{ \lambda_{k}(\e) \}_{k\geqslant1}$ be the eigenvalues satisfying $0 < \lambda_{1}(\e) < \lambda_{2}(\e) \leqslant \lambda_{3}(\e) \leqslant \cdots \to +\infty$ and $\{ \Phi_{k,\e} \}_{k\geqslant1}$ be the associated eigenfunctions in \eqref{P} which are assumed to be normalized so that 
\begin{equation*}
\int_{\Omega_{\e}} \abs{\Phi_{k,\e}}^{2}  \, dx = 1. 
\end{equation*}

Since $Q_{\e} = Q_{\e}(x)$ $(x \in \Omega_{\e})$ is a piecewise constant function, we can rewrite \eqref{eigenfunction} as follows:
\begin{equation}\label{pb2}
\begin{cases}
- \Delta \Phi_{1} = \lambda \Phi_{1} \quad &\text{in} \,\, \Omega, \\
- \sig \Delta \Phi_{2} = \lambda \Phi_{2} \quad &\text{in} \,\, \Sigma_{\e}, \\
\Phi_{1} = \Phi_{2} \quad &\text{on} \,\, \Gamma, \\
\vspace{0.15cm}
\displaystyle \frac{\partial \Phi_{1}}{\partial \nu_{\Gamma}} = \sig \frac{\partial \Phi_{2}}{\partial \nu_{\Gamma}} \quad &\text{on} \,\, \Gamma, \\
\displaystyle \Phi_{2} = 0 \quad &\text{on} \,\, \pa \Omega_{\e}.
\end{cases} 
\end{equation}
Here $\Phi_{1}$ and $\Phi_{2}$ are the restriction of the eigenfunction $\Phi$ on $\Omega$ and $\Sigma_{\e}$, respectively. The fourth equality in \eqref{pb2} is usually called {\em transmission condition}, which can be interpreted as the continuity of the flux through the interface $\Gamma$ in \eqref{P}. 
\begin{figure}[H]
\centering
\includegraphics[width=6cm]{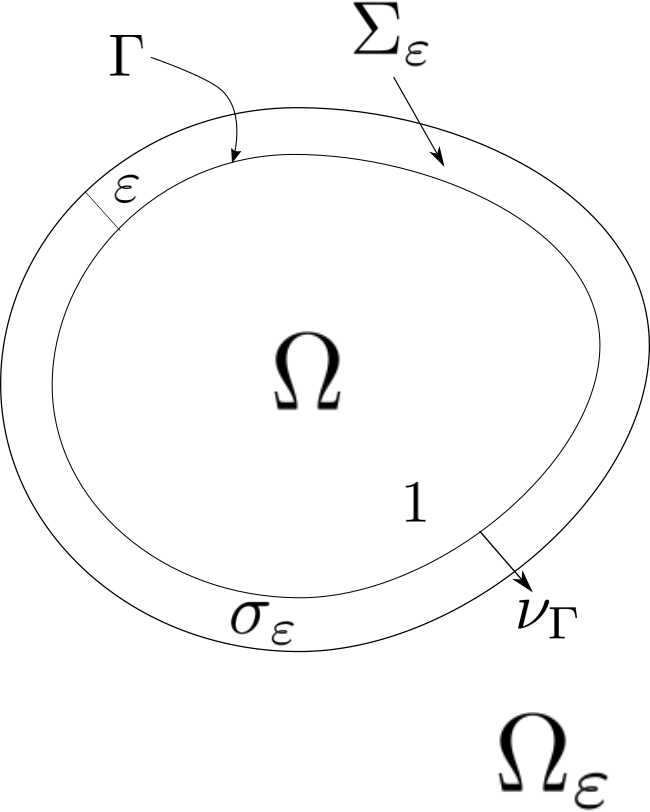}
\caption{Problem setting}
\end{figure}

The purpose of this paper is to study the asymptotic behavior for the principal eigenvalue $\lambda_{1}(\e)$ as $\e \to 0$. In particular, our aim is to show how the geometric shape of the interface $\Gamma$ affects the asymptotic behavior for the principal eigenvalue $\lambda_{1}(\e)$.  In what follows, let the principal eigenfunction $\Phi_{1,\e}$ be denoted by $\Phi_{\e}$ for the sake of simple notation. 

The study of two-phase eigenvalue problems arise in the study of the material science of composite media. In particular, the problems dealt with in this paper are called reinforcement problems or coating problems, and are related to vibration frequencies of composite materials or coating of composite materials with thermal insulation. 

This type of the two-phase eigenvalue problem was first studied by Friedman\cite{Fr}. He considered the two-phase eigenvalue problem for the principal eigenvalue of some elliptic operators in the case $\lim_{\e \to 0} \sig / \e = \alpha$ and $\lim_{\e \to 0} \sig / \e = 0$. His method is based on $H^{2}$-estimate of the eigenfunction. Rosencrans--Wang\cite{RW} generalized Friedman's results to all eigenvalues in the case $\lim_{\e \to 0} \sig / \e = 0$. They only used $H^{1}$-estimate of eigenfunctions which is easily obtained by the variational characterization of the eigenvalues. Regarding other two-phase eigenvalue problems in this direction, we refer to \cite{GLNP}\cite{JKo}\cite{P}. 

In this paper, we treat the case $\lim_{\e \to 0} \sig / \e = \alpha$ and focus on a refinement of Friedman's result. 
Friedman proved the following theorem:
\begin{thm}[Friedman]\label{thm0}
Let $\lambda_{1}(\e)$ be the principal eigenvalue of the eigenvalue problem \eqref{P}. Then we have  
\begin{align*}
\lambda_{1}(\e) &= \mu_{1} + o(1) \,\,\, \text{as} \,\,\, \e \to 0, \\
\Phi_{\e} &\to w_{1} \,\,\, \text{weakly} \,\,\, \text{in} \,\,\, H^{2}(\Omega),  
\end{align*}
where $\mu_{1}$ is the principal eigenvalue and $w_{1}$ is the principal eigenfunction of the following Robin eigenvalue problem:
\begin{equation*}
\begin{cases}
- \Delta w = \mu w \,\, &\text{in} \,\, \Omega, \\
\alpha w + \dfrac{\pa w}{\pa \nu_{\Gamma}} = 0 \, &\text{on} \,\, \Gamma.
\end{cases}
\end{equation*}
\end{thm}
Theorem \ref{thm0} implies that the condition $\lim_{\e \to 0} \sig / \e = \alpha$ affects the boundary condition, which becomes the Robin boundary condition. 

We derive a more precise asymptotic behavior for the principal eigenvalue. In the following we mention the main result of this paper. 
\begin{thm}\label{thm1}
Let $\lambda_{1}(\e)$ be the principal eigenvalue of the eigenvalue problem \eqref{P}. Then we have the asymptotic behavior 
\begin{equation*}
\lambda_{1}(\e) = \mu_{1} - \e \int_{\Gamma} \left( \alpha H + \frac{\mu_{1}}{3} \right)w^{2}_{1} \G + o(\e) \,\,\, \text{as} \,\,\, \e \to 0, 
\end{equation*}
where $H$ is the mean curvature defined as the sum of the principle curvatures of $\Gamma$.   
\end{thm}
From Theorem \ref{thm1}, we see that the effect of the geometric shape of the interface $\Gamma$ appears in the second term of the asymptotic behavior for the principal eigenvalue. 

The outline of the proof of Theorem \ref{thm1} is as follows: first, we derive an upper bound of the principal eigenvalue by using a variational approach which is based on \cite{RW}. Next, we derive a lower bound of the principal eigenvalue by using the upper bound and the Fourier expansion with respect to eigenfunctions of a Robin eigenvalue problem. Then we have the asymptotic behavior for the principal eigenvalue. Once the asymptotic behavior is obtained, we can use it to show an $L^{2}$-estimate for the tangential components of the principal eigenfunction in $\Sigma_{\e}$. This is necessary to control the behavior of the principal eigenfunction in the thin layer $\Sigma_{\e}$. By using this estimate, $H^{2}$-estimate, and transmission condition, we finally prove Theorem \ref{thm1}. 

The following sections are organized as follows: in section \ref{pre}, we give some geometric preliminaries concerning the thin layer $\Sigma_{\e}$. In section \ref{firstasmp}, we prove the asymptotic behavior for the principal eigenvalue. In section \ref{prth1}, based on the results of section \ref{firstasmp}, we prove Theorem \ref{thm1}. 

\section{Geometric preliminaries}\label{pre}
We present some geometric preliminaries of thin layer $\Sigma_{\e}$. Every $x \in \Sigma_{\e}$ can be represented by 
\begin{equation}\label{x}
x = \xi + \tau \nu_{\Gamma}(\xi), \quad \xi \in \Gamma, \,\, 0 < \tau < \e.
\end{equation}
We introduce a local coordinate system $(\xi_{1}, \xi_{2}, \cdots , \xi_{n-1}, \xi_{n}) = (\xi_{1}, \xi_{2}, \cdots , \xi_{n-1}, \tau)$ for $\Gamma \times (0, \e)$ and let $g = \left( g_{ij}(\xi,\tau) \right)$ denote the metric tensor associated with it. Then from \eqref{x}, $g_{ij}(\xi,\tau)$ is given by 
\begin{equation}\label{asymp1}
g_{ij}(\xi,\tau) = \begin{cases}
g_{0,ij}(\xi) + \tau \widetilde{g}_{0,ij}(\xi) + \tau^{2}\widehat{g}_{0,ij}(\xi) &\text{if} \quad 1\leqslant i,j \leqslant n-1, \\
0 &\text{if} \quad i = n, j \neq n \,\, \text{or}\,\, i \neq n, j = n, \\
1 &\text{if} \quad i , j = n,
\end{cases}
\end{equation}
where $g_{0} = \left( g_{0, ij}(\xi) \right)$ denotes the Riemannian metric associated with the local coordinates $(\xi_{1}, \xi_{2}, \cdots , \xi_{n-1})$ and we put
\begin{equation*}
\widetilde{g}_{0,ij} = \left(\frac{\pa}{\pa \xi_{i}}, \frac{\pa \nu_{\Gamma}}{\pa \xi_{j}}\right) + \left(\frac{\pa}{\pa \xi_{j}}, \frac{\pa \nu_{\Gamma}}{\pa \xi_{i}}\right), \quad 
\widehat{g}_{0,ij} = \left(\frac{\pa \nu_{\Gamma}}{\pa \xi_{i}}, \frac{\pa \nu_{\Gamma}}{\pa \xi_{j}}\right).
\end{equation*}
Here $\pa / \pa \xi_{i}$ and $\pa / \pa \xi_{j}$ are tangent vectors on $\xi \in \Gamma$ and $(\cdot , \cdot)$ is the Euclidean inner product. Let $(b_{ij})_{1\leqslant i,j \leqslant n-1}$ denote the coefficients of the second fundamental form on $\Gamma$. In the local coordinate, $b_{ij} = \left(\pa^{2} / \pa \xi_{i} \pa \xi_{j}, \nu_{\Gamma} \right)$.
By the definition of $\widetilde{g}_{0,ij}$, we have $\widetilde{g}_{0,ij} = -2b_{ij}$. 
Also we denote the inverse matrix of $(g_{ij})$ by $(g^{ij})$ and put $G = \det (g_{ij})$. By using this local coordinates we can express the norm of the gradient of $\Phi$ as follows: 
\begin{equation}
\abs{\nabla_{x} \Phi}^{2} = \sum^{n}_{i,j = 1} g^{ij} \frac{\pa \Phi}{\pa \xi_{i}} \frac{\pa \Phi}{\pa \xi_{j}} = \abs{\nabla_{\tan} \Phi}^{2} + \left( \frac{\pa \Phi}{\pa \tau} \right)^{2}, \label{norm} 
\end{equation}
where $\abs{\nabla_{\tan} \Phi}^{2} = \sum^{n-1} _{i,j=1} g^{ij} \pa \Phi / \pa \xi_{i} \pa \Phi / \pa \xi_{j}$. Moreover, by \eqref{asymp1} we can obtain the following asymptotic formula for $\sqrt{G}$: 
\begin{equation}
\sqrt{G(\xi,\tau)} = \sqrt{G(\xi,0)}(1 - H(\xi) \tau) + O(\tau^{2}) \,\,\, \text{as} \,\,\, \tau \to 0, \label{formula0}
\end{equation}
where $H(\xi)$ is the mean curvature at $\xi \in \Gamma$ with respect to $\nu_{\Gamma}$ (defined as the sum of the principle curvatures of $\Gamma$). 
The asymptotic formula \eqref{formula0} will play an important role in obtaining the asymptotic behavior for the principal eigenvalue $\lambda_{1}(\e)$. 
For the details about the geometric property of a thin layer, see \cite{JK}\cite{S}\cite{Y} and the references given there. 

In the following sections, $\sqrt{G(\xi, \tau)}$ will be denoted by $\sqrt{G_{\tau}}$ for simplicity and C will be used to represent any positive constant independent of $\e$. The same letter $C$ will be used to denote different constants.

\section{Asymptotic behavior for $\lambda_{1}(\e)$}\label{firstasmp}
\subsection{Upper bound of $\lambda_{1}(\e)$}
By the $\min$-$\max$ principle, 
\begin{equation}\label{minmax}
\lambda_{1}(\e) = \inf_{u \in H^{1}_{0}(\Omega), \, u \neq 0} \frac{\displaystyle \int_{\Omega} \abs{\nabla u}^{2} dx + \sig \int_{\Sigma_{\e}} \abs{\nabla u}^{2} dx}{\displaystyle \int_{\Omega_{\e}} \abs{u}^{2} dx}. 
\end{equation}
We construct a test function in order to estimate the principal eigenvalue $\lambda_{1}(\e)$. We extend the normalized Robin principal eigenfunction $w_{1} = w_{1}(x)$ ($x \in \Omega$) along $\nu_{\Gamma}$ to $\Sigma_{\e}$ by setting $w_{1}(\xi, \tau) = w_{1}(\xi)$ for every $\xi \in \Gamma$. 
Also we put
\begin{equation*}
\phi(x) = \begin{cases}
1 \quad \text{in} \,\, \Omega, \\
1 - \dfrac{\tau}{\e} \quad \text{in} \,\, \overline{\Sigma}_{\e}.
\end{cases}
\end{equation*}
Taking $\tilde{u} = w_{1} \phi$ as a test function in \eqref{minmax}, we obtain 
\begin{equation*}
\lambda_{1}(\e) \leqslant \frac{\displaystyle \int_{\Omega} \abs{\nabla \tilde{u}}^{2} dx + \sig \int_{\Sigma_{\e}} \abs{\nabla \tilde{u}}^{2} dx}{\displaystyle \int_{\Omega} \abs{\tilde{u}}^{2} dx} = \frac{\displaystyle \int_{\Omega} \abs{\nabla w_{1}}^{2} dx + \sig \int_{\Sigma_{\e}} \abs{\nabla(w \phi)}^{2} dx}{\displaystyle \int_{\Omega} \abs{w_{1}}^{2} dx + \int_{\Sigma_{\e}} \abs{w_{1} \phi}^{2} dx}. 
\end{equation*}

By using the normalization $\int_{\Omega} \abs{w_{1}}^{2} dx = 1$ and $\int_{\Sigma_{\e}} \abs{w_{1} \phi}^{2} dx = O(\e)$, we have 
\begin{equation*}
\int_{\Omega} \abs{w_{1}}^{2} dx + \int_{\Sigma_{\e}} \abs{w_{1} \phi}^{2} dx = 1 + O(\e). 
\end{equation*} 
Also we have $\nabla w_{1} \cdot \nabla \phi = 0$ since $w_{1}$ and $\phi$ only depend on $\xi$ and $\tau$ in $\Sigma_{\e}$, respectively. Hence,    
\begin{align*}
\sig \int_{\Sigma_{\e}} \abs{\nabla(w_{1} \phi)}^{2} dx &= \sig \int_{\Sigma_{\e}} \left( \phi^{2} \abs{\nabla w_{1}}^{2} + 2 \nabla w_{1} \cdot \nabla \phi + w_{1}^{2} \abs{\nabla \phi}^{2} \right) dx \\
&= \sig \int_{\Sigma_{\e}} \phi^{2} \abs{\nabla w_{1}}^{2} dx + \sig \int_{\Sigma_{\e}} w_{1}^{2} \abs{\nabla \phi}^{2} dx \\
&= \sig \int_{\Sigma_{\e}} w_{1}^{2} \abs{\nabla \phi}^{2} dx + O(\e^{2}). 
\end{align*}
We note that $\nabla \phi = \frac{\pa \phi}{\pa \tau} \nu_{\Gamma} = - \nu_{\Gamma}/\e$. By using the asymptotic formula \eqref{formula0}, we get 
\begin{align*}
\sig \int_{\Sigma_{\e}} w_{1}^{2} \abs{\nabla \phi}^{2} dx &= \alpha \e \intp w_{1}(\xi)^{2} \cdot \abs{\frac{-\nu_{\Gamma}}{\e}}^{2} \Gtau \\ 
&= \alpha \e \cdot \frac{1}{\e^{2}}\intp w_{1}(\xi)^{2} (1 + O(1) \tau) \Gt \\
&= \alpha \intg w_{1}^{2} \G + O(\e).
\end{align*}
Thus, 
\begin{align*}
\lambda_{1}(\e) &\leqslant \frac{\displaystyle \int_{\Omega} \abs{\nabla \tilde{u}}^{2} dx + \sig \int_{\Sigma_{\e}} \abs{\nabla \tilde{u}}^{2} dx}{\displaystyle \int_{\Omega} \abs{\tilde{u}}^{2} dx} \\
&\leqslant \frac{\displaystyle \int_{\Omega} \abs{\nabla w_{1}}^{2} dx + \alpha \intg w_{1}^{2} \G + C\e}{1 - C\e} \leqslant \mu_{1} + C\e. 
\end{align*}
Therefore we obtain the following upper bound of the principal eigenvalue $\lambda_{1}(\e)$:  
\begin{equation}\label{upper bound}
\lambda_{1}(\e) \leqslant \mu_{1} + C\e.  
\end{equation}

\subsection{Lower bound of $\lambda_{1}(\e)$}
Recall the weak form \eqref{eigenfunction}: for any $\varphi \in H^{1}_{0}(\Omega)$, 
\begin{equation*}
\int_{\Omega} \nabla \Phi_{\e} \cdot \nabla \varphi \, dx + \sig \int_{\Sigma_{\e}} \nabla \Phi_{\e} \cdot \nabla \varphi \, dx = \lambda_{1}(\e) \int_{\Omega_{\e}} \Phi_{\e} \varphi \, dx. 
\end{equation*}
First of all, we mention that we can get the following $H^{1}$ and $H^{2}$-estimates of the principal eigenfunction $\Phi_{\e}$ by using the upper bound \eqref{upper bound}. 
\begin{lem}\label{key lemma0}
The principal eigenfunction $\Phi_{\e}$ satisfies 
\begin{align}
&\int_{\Omega} \abs{\nabla \Phi_{\e}}^{2} \, dx + \sig \int_{\Sigma_{\e}} \abs{\nabla \Phi_{\e}}^{2} \, dx \leqslant C, \label{H^{1}} \\
&\int_{\Omega} \abs{D^{2} \Phi_{\e}}^{2} \, dx + \sig \int_{\Sigma_{\e}} \abs{D^{2} \Phi_{\e}}^{2} \, dx \leqslant C \label{H^{2}}
\end{align}
for a positive constant $C$ independent of $\e$. 
\end{lem}
\begin{proof}
It is easy to show the estimate \eqref{H^{1}} by taking $\varphi = \Phi_{\e}$ in \eqref{eigenfunction} and using the upper bound \eqref{upper bound}. 
The $H^{2}$-estimate \eqref{H^{2}} derives from a boundary estimate on $\Gamma$, which was first established by Brezis--Caffarelli--Friedman\cite{BCF} in the case of two-phase elliptic equations. Friedman\cite{Fr} proved the $H^{2}$-estimate \eqref{H^{2}} by using a similar method. Thus we omit this proof.
\end{proof}
We take any $\zeta \in C^{1}(\overline{\Omega})$. Let us extend $\zeta$ along $\nu_{\Gamma}$ to $\Sigma_{\e}$ by $\zeta(\xi,\tau) = \zeta(\xi)$ for every $\xi \in \Gamma$. We take $\varphi = \zeta\phi$ as a test function in \eqref{eigenfunction}, then we have
\begin{align*}
&\int_{\Omega} \nabla \Phi_{\e} \cdot \nabla \zeta \, dx + \sig \int_{\Sigma_{\e}} \phi \nabla \Phi_{\e} \cdot \nabla \zeta \, dx + \sig \int_{\Sigma_{\e}} \zeta \nabla \Phi_{\e} \cdot \nabla \phi \, dx \\
&\qquad \qquad \qquad \qquad \qquad \qquad \qquad = \lambda_{1}(\e) \int_{\Omega} \Phi_{\e} \phi \zeta \, dx + \lambda_{1}(\e) \int_{\Sigma_{\e}} \Phi_{\e} \phi \zeta \, dx. 
\end{align*}
The second term on the left-hand side and the second term on the right-hand side are $O(\e)$. Indeed, for any $\zeta\in C^{1}(\overline{\Omega})$, by using the $H^{1}$-estimate \eqref{H^{1}} we have 
\begin{align*}
\abs{\sig \int_{\Sigma_{\e}} \phi \nabla \Phi_{\e} \cdot \nabla \zeta \, dx} &\leqslant \int_{\Sigma_{\e}} \abs{\sig^{1/2} \nabla \Phi_{\e}} \cdot \abs{\sig^{1/2} \nabla \zeta} dx \\
&\leqslant \left( \sig \int_{\Sigma_{\e}} \abs{\nabla \Phi_{\e}}^{2} dx \right)^{1/2} \left( \sig \int_{\Sigma_{\e}} \abs{\nabla \zeta}^{2} dx \right)^{1/2} \\
&\leqslant C\e. 
\end{align*}
By using the upper bound of $\lambda_{1}(\e)$, we also have
\begin{align*}
\abs{\lambda_{1}(\e) \int_{\Sigma_{\e}} \Phi_{\e} \phi \zeta \, dx} &\leqslant C \int_{\Sigma_{\e}} \abs{\Phi_{\e}} \abs{\zeta} dx \\
&\leqslant C \left( \int_{\Sigma_{\e}} \abs{\Phi_{\e}}^{2} dx \right)^{1/2} \left( \int_{\Sigma_{\e}} \abs{\zeta}^{2} dx \right)^{1/2}. 
\end{align*}
Now we need to estimate $\int_{\Sigma_{\e}} \abs{\Phi_{\e}}^{2} dx$. By the Dirichlet boundary condition on $\pa \Omega_{\e}$, we get
\begin{equation}\label{keyeq1}
\Phi_{\e}(\xi,\tau) = - \int^{\e}_{\tau} \dfrac{\pa \Phi_{\e}}{\pa \tau} ds. 
\end{equation}
This identity implies that
\begin{equation}\label{ineq1}
\abs{\Phi_{\e}}^{2} \leqslant \left( \int^{\e}_{0} \abs{\dfrac{\pa \Phi_{\e}}{\pa \tau}} ds \right)^{2} \leqslant \e \int^{\e}_{0} \abs{\dfrac{\pa \Phi_{\e}}{\pa \tau}}^{2} ds. 
\end{equation}
Thus we have 
\begin{equation}\label{ineq2}
\int_{\Sigma_{\e}} \abs{\Phi_{\e}}^{2} dx \leqslant \e \int_{\Sigma_{\e}} \left( \int^{\e}_{0} \abs{\dfrac{\pa \Phi_{\e}}{\pa \tau}}^{2} ds \right) dx \leqslant C \e \left( \sig \int_{\Sigma_{\e}} \abs{\nabla \Phi_{\e}}^{2} dx \right). 
\end{equation}
Therefore we obtain the following estimate: 
\begin{equation*}
\abs{\lambda_{1}(\e) \int_{\Sigma_{\e}} \Phi_{\e} \phi \zeta \, dx} \leqslant C\e. 
\end{equation*}
Note that $\nabla \phi = \frac{\pa \phi}{\pa \tau} \nu_{\Gamma} = - \nu_{\Gamma}/\e$ and, by using the asymptotic formula \eqref{formula0}, we have 
\begin{align*}
\sig \int_{\Sigma_{\e}} \zeta \nabla \Phi_{\e} \cdot \nabla \phi \, dx &= \sig \int_{\Sigma_{\e}} \zeta \nabla \Phi_{\e} \cdot \left( - \frac{\nu_{\Gamma}}{\e} \right) \, dx \\
&=  - \alpha \intp \zeta \frac{\pa \Phi_{\e}}{\pa \tau} \Gtau \\
&= - \alpha \intp \zeta \frac{\pa \Phi_{\e}}{\pa \tau} \left(1 + O(1) \tau \right) \Gt \\
&= \alpha \intg \Phi_{\e} \zeta \G + O(\e).  
\end{align*}
Therefore, we obtain
\begin{equation}\label{limiteq1}
\int_{\Omega} \nabla \Phi_{\e} \cdot \nabla \zeta dx + \alpha \intg \Phi_{\e} \zeta \G = \lambda_{1}(\e) \int_{\Omega} \Phi_{\e} \zeta dx + O(\e). 
\end{equation}

We consider the Fourier expansions of $\Phi_{\e}$ with respect to the orthonormal basis given by the eigenfunctions of the following Robin eigenvalue problem:  
\begin{equation}\label{robin}
\begin{cases}
- \Delta w = \mu w \,\, &\text{in} \,\, \Omega, \\
\alpha w + \dfrac{\pa w}{\pa \nu_{\Gamma}} = 0 \, &\text{on} \,\, \Gamma.
\end{cases}
\end{equation}
Let $\{ \mu_{k} \}_{k\geqslant1}$ be the eigenvalues corresponding to the problem \eqref{robin} ordered so that they satisfy $0 < \mu_{1} \leqslant \mu_{2} \leqslant \mu_{3} \leqslant \cdots \to +\infty$ and $\{ w_{k} \}_{k\geqslant1}$ be the associated eigenfunctions which are assumed to be normalized so that 
\begin{equation*} 
\int_{\Omega} \abs{w_{k}}^{2} dx = 1. 
\end{equation*}
Then $\Phi_{\e}$ admits the following the Fourier expansions in $H^{1}(\Omega)$: 
\begin{equation}\label{Fourier expansions}
\Phi_{\e} = \sum_{k \geqslant 1} c_{k}(\e) w_{k}, \quad c_{k} = \int_{\Omega} \Phi_{\e} w_{k} dx.
\end{equation}

Taking $\zeta = c_{1} w_{1}$ in \eqref{limiteq1} and using the orthogonality of the Robin eigenfunctions $\{ w_{k} \}_{k\geqslant1}$, we have 
\begin{equation}\label{above}
(c_{1})^{2} \mu = \lambda_{1}(\e) (c_{1})^{2} + O(\e). 
\end{equation}
From the estimate \eqref{above}, it will be sufficient to show the following lemma to get the lower bound of $\lambda_{1}(\e)$. 
\begin{lem}\label{key lemma1}
The following estimate holds:
\begin{equation}\label{lemma1}
c_{1}(\e) = 1 + o(1) \,\,\, \text{as} \,\,\, \e \to 0.
\end{equation}
\end{lem}
\begin{proof}
From Lemma \ref{key lemma0}, we obtain the $H^{1}$-boundedness of the principal eigenfunction $\Phi_{\e}$ in $\Omega$. 
Applying Rellich's Theorem, after passing to a subsequence, there exists $\widehat{\Phi} \in H^{1}(\Omega)$ such that $\Phi_{\e} \to \widehat{\Phi}$ strongly in $L^{2}(\Omega)$ and weakly in $H^{1}(\Omega)$. Moreover, for some nonnegative value $\widehat{\lambda}$ we also have $\lambda_{1}(\e) \to \widehat{\lambda}$ and $\widehat{\lambda} \leqslant \mu_{1}$.  If we let $\e \to 0$ in \eqref{limiteq1}, then 
\begin{equation}\label{limiteq2}
\int_{\Omega} \nabla \widehat{\Phi} \cdot \nabla \zeta dx + \alpha \intg \widehat{\Phi} \zeta \G = \widehat{\lambda} \int_{\Omega}\widehat{\Phi} \zeta dx. 
\end{equation} 
Thus $\widehat{\lambda}$ is a Robin eigenvalue and $\widehat{\Phi}$ is the corresponding Robin eigenfunction. It implies that $\mu_{1} \leqslant \widehat{\lambda}$. 
Therefore we obtain $\widehat{\lambda} = \mu_{1}$. Since $\mu_{1}$ is the principal eigenvalue, we have $\widehat{\Phi} = \pm w_{1}$. Also, since $\Phi_{\e}$ is chosen to be positive function, we get $\widehat{\Phi} = w_{1}$. By using the fact that $\Phi_{\e}$ converges to $\widehat{\Phi}$ strongly in $L^{2}(\Omega)$, we get the estimate $c_{1} = 1 + o(1)$ as $\e \to 0$.  
\end{proof}
From \eqref{above} and Lemma \ref{key lemma1} we have
\begin{equation}\label{lower bound}
\lambda_{1}(\e) \geqslant \mu_{1} - C\e.
\end{equation}
Combining the upper bound \eqref{upper bound} with the lower bound \eqref{lower bound}, we obtain
\begin{equation}\label{asymp}
\lambda_{1}(\e) = \mu_{1} + O(\e) \,\,\, \text{as} \,\,\, \e \to 0. 
\end{equation}

\section{Proof of Theorem \bf{\ref{thm1}}}\label{prth1}
First of all, we show the $L^{2}$ estimate for the tangential components of $\nabla \Phi_{\e}$. 
\begin{lem}\label{key lemma2}
The following estimate holds:
\begin{equation}\label{lemma2}
\sig \int_{\Sigma_{\e}} \abs{\nabla_{tan} \Phi_{\e}}^{2} dx = O(\e) \,\,\, \text{as} \,\,\, \e \to 0. 
\end{equation}
\end{lem}
\begin{proof}
By using \eqref{norm} we have 
\begin{align*}
\lambda_{1}(\e) &= \int_{\Omega} \abs{\nabla \Phi_{\e}}^{2} dx + \sig \int_{\Sigma_{\e}} \abs{\nabla \Phi_{\e}}^{2} dx \\
&= \int_{\Omega} \abs{\nabla \Phi_{\e}}^{2} dx + \sig \int_{\Sigma_{\e}} \abs{\nabla_{tan} \Phi_{\e}}^{2} dx + \sig \int_{\Sigma_{\e}} \abs{\frac{\pa \Phi_{\e}}{\pa \nu_{\Gamma}}}^{2} dx. 
\end{align*}
Now we estimate $\sig \int_{\Sigma_{\e}} \abs{\frac{\pa \Phi_{\e}}{\pa \nu_{\Gamma}}}^{2} dx$. By using the estimate \eqref{ineq1}, we obtain
\begin{equation*}
\alpha \intg \Phi_{\e}^{2} \G \leqslant \sig \int_{\Sigma_{\e}} \abs{\frac{\pa \Phi_{\e}}{\pa \nu_{\Gamma}}}^{2} dx + C\e. 
\end{equation*}
Therefore, 
\begin{align*}
\lambda_{1}(\e) &= \int_{\Omega} \abs{\nabla \Phi_{\e}}^{2} dx + \sig \int_{\Sigma_{\e}} \abs{\nabla_{tan} \Phi_{\e}}^{2} dx + \sig \int_{\Sigma_{\e}} \abs{\frac{\pa \Phi_{\e}}{\pa \nu_{\Gamma}}}^{2} dx \\
&\geqslant \int_{\Omega} \abs{\nabla \Phi_{\e}}^{2} dx + \sig \int_{\Sigma_{\e}} \abs{\nabla_{tan} \Phi_{\e}}^{2} dx + \alpha \intg \Phi_{\e}^{2} \G - C\e \\
&\geqslant \mu_{1} + \sig \int_{\Sigma_{\e}} \abs{\nabla_{tan} \Phi_{\e}}^{2} dx - C\e. 
\end{align*}
From the asymptotic behavior \eqref{asymp}, we get the following estimate: 
\begin{equation*}
\sig \int_{\Sigma_{\e}} \abs{\nabla_{tan} \Phi_{\e}}^{2} dx = O(\e) \,\,\, \text{as} \,\,\, \e \to 0. 
\end{equation*}
\end{proof}
The estimate in the following lemma will be useful to examine the behavior for $\Phi_{\e}$ in $\Sigma_{\e}$. Recall that $\Phi_{1,\e}$ and $\Phi_{2,\e}$ denote the restriction of $\Phi_{\e}$ on $\Omega$ and $\Sigma_{\e}$, respectively.  
\begin{lem}\label{key lemma3}
The following estimate holds: 
\begin{equation}\label{lemma3}
\norm{\Phi_{1,\e} + \frac{1}{\alpha} \frac{\pa \Phi_{1,\e}}{\pa \nu_{\Gamma}}}^{2}_{L^{2}(\Gamma)} = O(\e) \,\,\, \text{as} \,\,\, \e \to 0. 
\end{equation}
\end{lem}
\begin{proof}
For any $k \in (0,\e)$, 
\begin{equation*}
\frac{\pa \Phi_{2,\e}}{\pa \tau}(\xi,k) - \frac{\pa \Phi_{2,\e}}{\pa \tau}(\xi,0) = \int^{k}_{0} \frac{\pa^{2} \Phi}{\pa \tau^{2}} ds.  
\end{equation*}
By integrating from $0$ to $\tau$ we have 
\begin{equation}\label{formula}
\Phi_{2,\e}(\xi, \tau) = \Phi_{2,\e}(\xi,0) + \tau \frac{\pa \Phi_{2,\e}}{\pa \tau}(\xi,0) + \int^{\tau}_{0} dk \int^{k}_{0} \frac{\pa^{2} \Phi_{2,\e}}{\pa \tau^{2}} ds. 
\end{equation}
Due to the Dirichlet boundary condition on $\pa \Omega_{\e}$, we obtain 
\begin{equation*}
\Phi_{2,\e}(\xi,0) + \e \frac{\pa \Phi_{2,\e}}{\pa \nu_{\Gamma}}(\xi,0) = - \int^{\e}_{0} \int^{k}_{0} \frac{\pa^{2} \Phi_{2,\e}}{\pa \tau^{2}} ds dk. 
\end{equation*}
Integrating on $\Gamma$ and using transmission condition, we have 
\begin{equation*}
\intg \abs{\Phi_{1,\e}(\xi,0) + \frac{1}{\alpha} \frac{\pa \Phi_{1,\e}}{\pa \nu_{\Gamma}}(\xi,0)}^{2} \G \leqslant C \e \left( \sig \int_{\Sigma_{\e}} \abs{D^{2} \Phi_{2,\e}}^{2} dx \right). 
\end{equation*}
Therefore from Lemma \ref{key lemma0} we obtain 
\begin{equation*}
\norm{\Phi_{1,\e} + \frac{1}{\alpha} \frac{\pa \Phi_{1,\e}}{\pa \nu_{\Gamma}}}^{2}_{L^{2}(\Gamma)} = O(\e) \,\,\, \text{as} \,\,\, \e \to 0. 
\end{equation*}
\end{proof}

We are now going to derive a more precise asymptotic behavior for the principal eigenvalue $\lambda_{1}(\e)$. Recall that
\begin{align*}
&\int_{\Omega} \nabla \Phi_{\e} \cdot \nabla \zeta \, dx + \sig \int_{\Sigma_{\e}} \phi \nabla \Phi_{\e} \cdot \nabla \zeta \, dx + \sig \int_{\Sigma_{\e}} \zeta \nabla \Phi_{\e} \cdot \nabla \phi \, dx \\
&\qquad \qquad \qquad \qquad \qquad \qquad \qquad = \lambda_{1}(\e) \int_{\Omega} \Phi_{\e} \zeta \, dx + \lambda_{1}(\e) \int_{\Sigma_{\e}} \Phi_{\e} \phi \zeta \, dx. 
\end{align*}
Noting that $\zeta$ only depends on $\xi$ in $\Sigma_{\e}$ and using Cauchy--Schwarz's inequality, we obtain 
\begin{align}
\abs{\sig \int_{\Sigma_{\e}} \phi \nabla \Phi_{\e} \cdot \nabla \zeta \, dx} &= \abs{\sig \int_{\Sigma_{\e}} \phi \nabla_{\tan} \Phi_{\e} \cdot \nabla_{\tan} \zeta \, dx} \notag \\
&\leqslant \left( \sig \int_{\Sigma_{\e}} \abs{\nabla_{\tan} \Phi_{\e}}^{2} dx\right)^{\frac{1}{2}} \left( \sig \int_{\Sigma_{\e}} \abs{\nabla_{\tan} \zeta}^{2} dx\right)^{\frac{1}{2}} = O(\e^{\frac{3}{2}}). \label{est1}
\end{align}
Moreover, 
\begin{align*}
\sig \int_{\Sigma_{\e}} \zeta \nabla \Phi_{\e} \cdot \nabla \phi \, dx &= \sig \int_{\Sigma_{\e}} \zeta \nabla \Phi_{\e} \cdot \left( - \frac{\nu_{\Gamma}}{\e} \right) \, dx \\
&=  - \alpha \intp \zeta \frac{\pa \Phi_{\e}}{\pa \tau} \Gtau \\
&= - \alpha \intp \zeta \frac{\pa \Phi_{\e}}{\pa \tau} \left(1 - H(\xi) \tau \right) \Gt + O(\e^{2}) \\
&= \alpha \intg \Phi_{\e} \zeta \G + \alpha \intp H(\xi) \zeta \tau \frac{\pa \Phi_{\e}}{\pa \tau} \Gt + O(\e^{2}). \\
\end{align*}
Using integration by parts and \eqref{formula}, we have
\begin{equation*}
\alpha \intp H(\xi) \zeta \tau \frac{\pa \Phi_{\e}}{\pa \tau} \Gt = - \alpha \e \intg  H \Phi_{\e} \zeta \G + O(\e^{2}). 
\end{equation*}
Thus we obtain 
\begin{equation}\label{est2}
\sig \int_{\Sigma_{\e}} \zeta \nabla \Phi_{\e} \cdot \nabla \phi \, dx = \alpha \intg \Phi_{\e} \zeta \G - \alpha \e \intg  H \Phi_{\e} \zeta \G + O(\e^{2}). 
\end{equation}
Furthermore, 
\begin{align*}
\int_{\Sigma_{\e}} \Phi_{\e} \phi \zeta \, dx &= \intp \Phi_{\e} \phi \zeta \Gtau \\
&= \intp \zeta \left( \Phi_{\e}(\xi,0) + \tau \frac{\pa \Phi_{\e}}{\pa \tau}(\xi,0) 
+ \int^{\tau}_{0} dk \int^{k}_{0} \frac{\pa^{2} \Phi_{\e}}{\pa \tau^{2}} ds \right) \\
&\qquad \qquad \qquad \times \left( 1 - \frac{\tau}{\e} \right) \times \left( 1 - H(\xi) \tau + O(\tau^{2}) \right) \Gt. 
\end{align*}
By direct computation, we have 
\begin{align*}
\int^{\e}_{0} \left( 1 - \frac{\tau}{\e} \right) d\tau = \frac{\e}{2}, \,\,
\int^{\e}_{0} \tau \left( 1 - \frac{\tau}{\e} \right) d\tau = \frac{\e^{2}}{6}, \,\,
\int^{\e}_{0} \tau^{2} \left( 1 - \frac{\tau}{\e} \right) d\tau = \frac{\e^{3}}{12}.   
\end{align*}
Thus we obtain 
\begin{align*}
&\intp \zeta \left( \Phi_{\e}(\xi,0) + \tau \frac{\pa \Phi_{\e}}{\pa \tau}(\xi,0) 
+ \int^{\tau}_{0} dk \int^{k}_{0} \frac{\pa^{2} \Phi_{\e}}{\pa \tau^{2}} ds \right) \\
&\qquad \qquad \qquad \times \left( 1 - \frac{\tau}{\e} \right) \times \left( 1 - H(\xi) \tau + O(\tau^{2}) \right) \Gt \\
&= \frac{\e}{2} \intg \zeta \Phi_{2,\e}(\xi,0) \G + \frac{\e^{2}}{6} \intg \zeta \frac{\pa \Phi_{2,\e}}{\pa \nu_{\Gamma}}(\xi,0) \G + O(\e^{2}) \\
&= \frac{\e}{2} \intg \zeta \Phi_{1,\e}(\xi,0) \G + \frac{\e}{6\alpha} \intg \zeta \frac{\pa \Phi_{1,\e}}{\pa \nu_{\Gamma}}(\xi,0) \G + O(\e^{2}), 
\end{align*}
where we used the transmission condition. 
From Lemma \ref{key lemma3} we have 
\begin{align*}
&\frac{\e}{2} \intg \zeta \Phi_{1,\e}(\xi,0) \G + \frac{\e}{6\alpha} \intg \zeta \frac{\pa \Phi_{1,\e}}{\pa \nu_{\Gamma}}(\xi,0) \G + O(\e^{2}) \\
&= \frac{\e}{2} \intg \zeta \Phi_{1,\e}(\xi,0) \G - \frac{\e}{6} \intg \zeta \Phi_{1,\e}(\xi,0) \G + O(\e^{\frac{3}{2}}) \\
&= \frac{\e}{3} \intg \zeta \Phi_{1,\e}(\xi,0) \G + O(\e^{\frac{3}{2}}). 
\end{align*}
Therefore
\begin{equation*}
\int_{\Sigma_{\e}} \Phi_{\e} \phi \zeta \, dx = \frac{\e}{3} \intg \zeta \Phi_{1,\e}(\xi,0) \G + O(\e^{\frac{3}{2}}).
\end{equation*}
By using the above and the asymptotic behavior \eqref{asymp}, we obtain the following estimate: 
\begin{equation}\label{est3}
\lambda_{1}(\e) \int_{\Sigma_{\e}} \Phi_{\e} \phi \zeta \, dx = \e \left( \frac{\mu_{1}}{3} \intg \zeta \Phi_{1,\e}(\xi,0) \G \right) + O(\e^{\frac{3}{2}}). 
\end{equation}
Combining \eqref{est1}, \eqref{est2}, and \eqref{est3} we have 
\begin{align*}
&\int_{\Omega} \nabla \Phi_{\e} \cdot \nabla \zeta \, dx + \alpha \intg \Phi_{\e} \zeta \G - \alpha \e \intg H \Phi_{\e} \zeta \G \\
&\qquad \qquad \qquad \qquad \qquad \qquad \qquad = \lambda_{1}(\e) \int_{\Omega} \Phi_{\e} \zeta \, dx + \e \left( \frac{\mu_{1}}{3} \intg \Phi_{\e} \zeta \G \right) + O(\e^{\frac{3}{2}}). 
\end{align*}
By using the fact that $\norm{\Phi_{\e} - w_{1}}_{L^{2}(\Gamma)} \to 0$ as $\e \to 0$, we obtain 
\begin{equation*}
\intg \Phi_{\e} \zeta \G = \intg w_{1} \zeta \G + o(1) \,\,\, \text{as} \,\,\, \e \to 0. 
\end{equation*}
Taking $\zeta = c_{1} w_{1}$, we get 
\begin{equation*}
(c_{1})^{2} \mu - \alpha \e \intg H c_{1} w_{1}^{2} \G = \lambda_{1}(\e) (c_{1})^{2} + \e \left( \frac{\mu_{1}}{3} \intg c_{1} w_{1}^{2} \G \right) + o(\e). 
\end{equation*}
Dividing by $(c_{1})^{2}$ and using Lemma \ref{key lemma1}, we finally obtain the following more precise asymptotic behavior for $\lambda_{1}(\e)$: 
\begin{equation*}
\lambda_{1}(\e) = \mu_{1} - \e \intg \left( \alpha H +\frac{\mu_{1}}{3} \right) w_{1}^{2} \G + o(\e) \,\,\, \text{as} \,\,\, \e \to 0. 
\end{equation*}
The proof of Theorem \ref{thm1} is complete. 

\vspace{1.0\baselineskip}
{\bf Acknowledgments.} 
The author would like to thank Professor Shigeru Sakaguchi (Tohoku University) for many stimulating discussions. Also the author would like to thank Lorenzo Cavallina (Tohoku University) for his warm encouragement.


\begin{thebibliography}{100}
\bibitem{BCF} H. Brezis, L. Caffarelli, A. Friedman, Reinforcement problems for elliptic equations and variational inequalities, Ann. Mat. Pura Appl., {\bf 123} (1980), 219--246.

\bibitem{Fr} A. Friedman, Reinforcement of the principal eigenvalue of an elliptic operator, Arch. Rational Mech. Anal. {\bf 73} (1980), no.1, 1--17.

\bibitem{GLNP} D. G\'{o}mez, M. Lobo, S.A. Nazarov, E. P\'{e}rez, Spectral stiff problems in domains surrounded by thin bands: Asymptotic and uniform estimates for eigenvalues
, J. Math. Pures Appl. {\bf 85} (2006), 598--632.



\bibitem{JKo} S. Jimbo, S. Kosugi, Approximation of eigenvalues of elliptic operators with discontinuous coefficients, Comm. Partial Differential Equations {\bf 28} (2003), 1303--1323.

\bibitem{JK} S. Jimbo, K. Kurata, Asymptotic behavior of eigenvalues of the Laplacian with the mixed boundary condition and its application, Indiana Univ. Math. J. {\bf 63} (2016), 867--898.

\bibitem{P} G.P. Panasenko, Asymptotics of the solutions and eigenvalues of elliptic equations with strongly varying coefficients, Soviet Math. Dokl. {\bf 21} (1980), 942--947.

\bibitem{RW} S. Rosencrans, X. Wang, Suppression of the Dirichlet Eigenvalues of a Coated Body, SIAM J. Appl. Math. {\bf 66} (2006), No.6, 1895--1916.

\bibitem{S} M. Schatzman, On the eigenvalues of the Laplace operator on a thin set with Neumann boundary conditions, Appl. Anal. {\bf 61} 
(1996), No.3--4, 293--306.

\bibitem{Y} T. Yachimura, Two-phase eigenvalue problem on thin domains with Neumann boundary condition, arxiv:1706.05027, (2017).

\end{thebibliography}
\end{document}